\documentclass[10pt,oneside,leqno]{amsart}
\usepackage{amsxtra}
\usepackage{amsopn}
\usepackage{color}
\usepackage{amsmath,amsthm,amssymb}
\usepackage{amscd}
\usepackage{amsfonts}
\usepackage[utf8x]{inputenc}
\usepackage{latexsym}
\usepackage{verbatim}
\usepackage{booktabs}
\usepackage{caption}
\usepackage{soul} % para tachar palabras

\theoremstyle{plain}
\newtheorem{theorem}{Theorem}[section]
\newtheorem{definition}[theorem]{Definition}
\newtheorem{lemma}[theorem]{Lemma}
\newtheorem{proposition}[theorem]{Proposition}
\newtheorem{corollary}[theorem]{Corollary}
\newtheorem{remark}[theorem]{Remark}

\newtheorem{example}[theorem]{Example}

\newtheorem{remark-question}[section]{Remark-Question}

\newcommand\C{{\mathbb C}}

\newcommand\N{{\mathbb N}}

\newcommand\R{{\mathbb R}}

\newcommand\g{{\gamma}}
\newcommand\frg{{\mathfrak g}}
\newcommand\frh{{\mathfrak h}}

\newcommand\frz{{\mathfrak z}}
\newcommand\aff{\mathfrak {aff}}
\newcommand\gc{\frg_\mathbb{C}}
\newcommand{\alt}{\raise1pt\hbox{$\bigwedge$}}

\allowdisplaybreaks[1]

\newcommand{\tr}{\operatorname{tr}}
\newcommand{\ad}{\operatorname{ad}}

\begin{document}

\begin{abstract}
 Let $\frg$ a $2n$-dimensional unimodular Lie algebra equipped with a Hermitian structure $(J,F)$ such that the complex structure $J$ is abelian and the fundamental form $F$ is balanced.  We prove that the holonomy group of the associated Bismut connection reduces to a subgroup of $SU(n-k)$, being $2k$ the dimension of the center of~$\frg$. We determine conditions that allow a unimodular Lie algebra to admit this particular type of structures. Moreover, we give methods to construct them in arbitrary dimensions and classify them if the Lie algebra is 8-dimensional and nilpotent. \end{abstract}

\title{Abelian balanced Hermitian structures on unimodular Lie algebras}

%\date{\today}

\author{Adri\'an Andrada}
\address[Andrada]{FaMAF-CIEM, Universidad Nacional de C\'{o}rdoba, Ciudad Universitaria, X5000HUA
C\'{o}rdoba, Argentina} \email{andrada@famaf.unc.edu.ar}

\author{Raquel Villacampa}
\address[Villacampa]{Centro Universitario de la Defensa Zaragoza-I.U.M.A.,
Academia Gene\-ral Militar\\
Carretera de Huesca, s/n
50090 Zaragoza, Spain} \email{raquelvg@unizar.es}

\maketitle

\

\section{Introduction}

Bismut proved in \cite{Bis} that given any Hermitian structure $(J,F)$ on a $2n$-dimensional manifold 
$M$ there is a unique Hermitian connection with
totally skew-symmetric torsion~$T$ given by
$g(X,T(Y,Z))=JdF(X,Y,Z)=-dF(JX,JY,JZ)$, $g$ being the associated metric. This torsion connection,
denoted by $\nabla$, is known as the {\it Bismut connection} of $(J,F)$ and it can be derived from
the Levi-Civita connection $\nabla^g$ of the Riemannian metric $g$ by $\nabla = \nabla^g
+\frac{1}{2}T$, where $T$ is identified with the 3-form~$JdF$.  Since $\nabla$ is Hermitian, its 
restricted holonomy group is contained in the unitary group $U(n)$.  However, stronger reductions of 
the holonomy group are interesting.  For instance, \emph{Calabi-Yau with torsion}  structures, i.e, 
structures satisfying  that $Hol(\nabla)$ is contained in $SU(n)$, appear in heterotic string 
theory, related to the Strominger system in six dimensions. 

If $M$ is a nilmanifold, i.e. a compact quotient of a simply connected nilpotent Lie group $G$ by a 
lattice, Fino, Parton and Salamon proved in \cite{FPS} that for invariant Hermitian structures on 
$M$, that is, for Hermitian structures arising from ones on the Lie algebra $\frg$ of $G$, the 
Calabi-Yau with torsion condition is equivalent to the balanced condition. 
Balanced metrics belong to the class $\mathcal W_3$ in the well-known Gray-Hervella 
classification~\cite{GH} and they are characterized by $d^{*}F=0$, where $d^*$ is the 
co-differential. Equivalently, $dF^{n-1}=0$, where $2n$ is the dimension of the 
Hermitian manifold.

In the particular case of dimension six,  nilmanifolds admitting balanced Hermitian structures are classified in~\cite{U}. In~\cite{UV} a complete study of the balanced geometry on such manifolds is carried out and solutions to the Strominger system are also found.  Moreover in~\cite[Theorem 4.7]{UV} the 
associated holonomy groups for the Bismut connection are described.  It turns out that the holonomy 
group of the Bismut connection reduces to a proper subgroup of $SU(3)$ if and only if the complex 
structure is abelian.  We recall that a complex structure on a Lie algebra $\frg$ is {\it abelian} if $[JX,JY]=[X,Y]$ for all $X,Y \in \frg$. These complex structures have interesting applications in differential geometry. For instance, a pair of anticommuting abelian complex structures on $\frg$ gives rise to an invariant weak HKT structure on $G$ (see \cite{cqg} and \cite{gp}). It has been shown in \cite{CF} that the Dolbeault cohomology of a nilmanifold with an abelian complex structure can be computed algebraically. Also, deformations of abelian complex structures on nilmanifolds have been studied in
\cite{cfp, MPPS}.

The main goal of this paper is to prove reductions of the holonomy group of the Bismut 
connection to $SU(n)$ in the case of Hermitian structures $(J, F)$, with $J$ abelian and $F$ 
balanced, on unimodular, not necessarily nilpotent, $2n$-dimensional Lie algebras  $\frg$. Moreover, 
if the center of $\frg$ is non-trivial, then the holonomy reduces to a proper subgroup of $SU(n)$. 

The paper is structured as follows.  Section 2 is devoted to study \emph{abelian balanced Hermitian 
structures} on unimodular Lie algebras, that is Hermitian structures $(J, F)$ where $J$ is abelian 
and $F$ is balanced.  We obtain equivalent conditions for a Hermitian structure to be abelian 
balanced in terms of the structure constants of the Lie algebra.

In Section 3 we focus our attention on the reduction of the holonomy group of the Bismut 
connection. The main result is Theorem~\ref{holonomia} stating that if 
$\frg$ is a unimodular Lie algebra of dimension $2n$ endowed with an abelian balanced 
Hermitian structure $(J,F)$ and $\dim \mathfrak z=2k$, where $\mathfrak z$ stands for the center of 
$\frg$, then the holonomy group of the associated Bismut connection is contained in $SU(n-k)$.  
 
In Section 4 we study the particular case of nilpotent Lie algebras.  In particular we prove that 
the product $\frh_{2n+1}\times \R^{2k+1}$ of a real Heisenberg Lie algebra by an odd number of 
copies of $\R$ admits abelian balanced Hermitian structures if and only if $n\geq 2$.  Moreover, the 
holonomy group of the Bismut connection reduces to $U(1)$, which provides the strongest reduction of 
the holonomy.  We also classify nilpotent Lie algebras in dimension 8 admitting abelian balanced 
Hermitian structures.

Finally, Section 5 is devoted to construct new examples of abelian balanced Hermitian structures 
on unimodular Lie algebras of any even dimension.

%%%%%%%%%%%%%%%%%%%%%%%%%%%%%%%%%%%%%%%%%%%
%%%%%%%%%%%%%%%%%%%%%%%%%%%%%%%%%%%%%%%%%%%%%

\section{Abelian balanced Hermitian structures on unimodular Lie algebras}

In this paper we are interested in abelian complex structures on Lie algebras endowed with
balanced Hermitian metrics.  We start by reviewing some known aspects about these topics.

A \emph{complex structure} on a Lie algebra $\frg$ is an endomorphism $J$ of $\frg$ satisfying $J^2
= -Id$ together with the vanishing of the Nijenhuis bilinear form with values in $\frg$, 
\[ N(X,Y) = [JX, JY] - J[JX, Y] - J[X, JY] - [X,Y],\] 
where $X, Y\in\frg$. This is equivalent to saying that the $i$-eigenspace $\frg_{1,0}$ of $J$ in
$\frg_{\mathbb C} = \frg\otimes_{\mathbb R}\mathbb C$ is a complex subalgebra of $\frg_{\mathbb C}$.
The complex structure $J$ is said of \emph{abelian type} if $[JX,JY]=[X,Y]$ for any $X, Y\in\frg$, 
or equivalently $\frg_{1,0}$ is abelian. 

In terms of the dual of $\gc$,  if we denote by $\frg^{1,0}$ and $\frg^{0,1}$ the eigenspaces
corresponding to the eigenvalues $\pm i$ of $J$ as an endomorphism of~$\gc^*$, respectively, then the
decomposition $\gc^*=\frg^{1,0}\oplus\frg^{0,1}$ induces a natural bigraduation on
the complexified exterior algebra $\bigwedge^* \,\gc^* =\oplus_{p,q}
\bigwedge^{p,q}(\frg^*)=\oplus_{p,q} \bigwedge^p(\frg^{1,0})\otimes
\bigwedge^q(\frg^{0,1})$. If $d$ denotes the usual Chevalley-Eilenberg differential of the Lie
algebra, we shall also denote by $d$ its extension to the complexified exterior algebra,
i.e. $d\colon \bigwedge^* \gc^* \longrightarrow \bigwedge^{*+1} \gc^*$. It is well known that the
endomorphism $J$ is a complex structure if and only if $d(\frg^{1,0})\subset
\bigwedge^{2,0}(\frg^*)\oplus \bigwedge^{1,1}(\frg^*)$.  Abelian complex structures are
characterized by the condition $d(\frg^{1,0})\subset \bigwedge^{1,1}(\frg^*)$. Some well-known properties of this type of structures are the following: the center $\frz$ of $\frg$ is 
$J$-invariant, the commutator ideal $\frg'=[\frg,\frg]$ is abelian (and therefore $J\frg'$ is also abelian), and $\operatorname{codim}\frg'\geq 2$ whenever $\dim\frg\geq 4$ (see for instance \cite{ABD1, BD}).

\

A {\it Hermitian structure} on $\frg$ is a pair $(J,g)$, where $J$ is a complex structure on $\frg$
and $g$ is an inner product on $\frg$ compatible with $J$ in the usual sense, i.e.
$g(\cdot,\cdot)=g(J\cdot,J\cdot)$. The associated {\it fundamental form} $F\in \bigwedge^2 \frg^*$
is defined by $F(X,Y)=g(X,JY)$. Fixed $J$, since $g$ and $F$ are mutually determined by each other,
we shall also denote the Hermitian structure $(J,g)$ by the pair $(J,F)$.   

A basis
$\{e_1,\ldots,e_{2n}\}$
of $\frg$ is said to be a \emph{$J$-adapted basis} if $Je_{2k-1}=-e_{2k}$ for all $k=1,\ldots,n$.  
Using the convention:  $$(J\eta)(X_1,\ldots, X_k) = (-1)^k
\eta(JX_1,\ldots, JX_k),\quad\text{where } \eta\in\alt^k \frg^*,\, X_j\in\frg,$$ it turns out that the dual 
basis $\{e^1,\ldots,e^{2n}\}$ also satisfies $Je^{2k-1}=-e^{2k}$.  Moreover, if we define 
$F=-\sum_{k=1}^n e^{2k-1}\wedge J e^{2k-1}$, then $\{e_1,\ldots, e_{2n}\}$ is orthonormal with 
respect to the Hermitian
metric $g(\cdot,\cdot) = F(J\cdot,\cdot)$. 

A basis $\{e_1,\ldots,e_{2n}\}$
of $\frg$ is said to be an \emph{adapted basis} for the Hermitian structure $(J, F)$ if in terms of the dual basis
\begin{equation}\label{adapted_basis}
Je^{2k-1}=-e^{2k},\quad F=-\sum_{k=1}^n e^{2k-1}\wedge J e^{2k-1}=\sum_{k=1}^n e^{2k-1}\wedge 
e^{2k}.
\end{equation}

The Hermitian structure $(J,F)$ is said to be {\it balanced} if $F^{n-1}$ is a closed
form or, equivalently, $d^*F=0$, where $d^*$ is the co-differential associated to the Riemannian
metric $g$ (see for instance \cite{M}).
\

In this section we will consider unimodular Lie algebras $\frg$ of dimension $2n$ with balanced Hermitian
structures $(J, F)$. Recall the following expression
for $d^*F$:
\[ d^*F(X)=-\sum_{i=1}^{2n} (\nabla^g_{e_i} F)(e_i,X), \]
where $X\in\frg$, $\nabla^g$ is the Levi-Civita connection associated to the Hermitian metric $g$ and $\{e_1,\ldots,
e_{2n}\}$ is any orthonormal basis of $\frg$. Taking into account the Koszul formula for
the Levi-Civita connection in the invariant setting, we obtain:
\begin{align*}
(\nabla^g_{e_i} F)(e_i,X) & = -F(\nabla^g_{e_i}e_i,X)-F(e_i,\nabla^g_{e_i}X) \\
                          & = -g(\nabla^g_{e_i}e_i,JX)+g(Je_i,\nabla^g_{e_i}X) \\
			  & = -\frac12
\left(2g([JX,e_i],e_i)+g([X,e_i],Je_i)+g([X,Je_i],e_i)-g([Je_i,e_i],X)     \right),
\end{align*}
and therefore
\[ d^*F(X)=\tr \ad_{JX}+\frac12 \left(-\tr(J\ad_X)+\tr(\ad_XJ) -\sum_{i=1}^{2n} g([Je_i,e_i],X)
\right),\]
or 
\[ d^*F(X)=\tr \ad_{JX}-\frac12 g\left( \sum_{i=1}^{2n} [Je_i,e_i],X \right).\]

Hence,
\[ d^*F=0 \Leftrightarrow \tr \ad_{JX}=\frac12 g\left( \sum_{i=1}^{2n} [Je_i,e_i],X \right) \]
for all $X\in\frg$. 
If, moreover, $\frg$ is unimodular, we have
\[ d^*F=0 \Leftrightarrow \sum_{i=1}^{2n} [Je_i,e_i] =0.\]
Observe that this condition is independent of the choice of the orthonormal basis.

\

We have proved the following

\begin{lemma}\label{J-bon}
Let $\frg$ be a $2n$-dimensional unimodular Lie algebra and let $(J,F)$ be a Hermitian structure on
$\frg$. Then $(J,F)$ is balanced if and
only if there exists an orthonormal basis $\{e_1,\ldots,e_{2n}\}$ of $\frg$ such that 
\begin{equation}\label{corchetes_balanced}
\sum_{i=1}^{2n} [Je_i,e_i]=0.
\end{equation}
In particular, if $\{e_1,\ldots, e_{2n}\}$ is a $J$-adapted basis 
satisfying~\eqref{corchetes_balanced}, then $F$ given by~\eqref{adapted_basis} is balanced.  
\end{lemma}

\

\begin{remark}
{\rm
 It follows immediately from Lemma \ref{J-bon} that if $\frg$ is unimodular and $J$ is a 
bi-invariant complex structure on $\frg$ (i.e., $J[X,Y]=[JX,Y]$ for all $X,Y\in\frg$), then any 
$J$-Hermitian metric $g$ on $\frg$ is balanced. This result was first proved in \cite{AG}.}
\end{remark}

\medskip

Let us consider a $J$-adapted basis $\{e_1,\ldots, e_{2n}\}$ of $\frg$ and let us express
$de^k=\sum_{i<j} c_{ij}^k e^{ij},$ where $e^{ij}$ stands for $e^i\wedge e^j$, or equivalently,
$[e_i,e_j] = -\sum_{k=1}^{2n}  c_{ij}^k e_k$, where $c_{ij}^k$ are the structure constants of
$\frg$. Here we are using the fact that $de^k(e_i,e_j) = -e^k([e_i, e_j])$.  In  this basis,
condition~\eqref{corchetes_balanced} is equivalent to the following system of equations in terms of
the structure constants of the Lie algebra:
\begin{equation}\label{const-estr-balanced}
\sum_{i=1}^n c_{2i-1,2i}^k =0,\quad k=1,\ldots,2n.
\end{equation}

As a consequence, we have the following equivalent condition for a Hermitian metric on a unimodular
Lie algebra to be balanced, which does not depend on any particular basis.

\begin{corollary}
Let $\frg$ be a $2n$-dimensional unimodular Lie algebra and let $(J,F)$ be a Hermitian structure on
$\frg$. Then $(J,F)$ is balanced if and
only if \begin{equation}\label{cond_balanced2}F^{n-1}\wedge d\alpha = 0,\end{equation} for any
$1$-form $\alpha\in\frg^*$.
\end{corollary}

\begin{proof}
Let us consider a basis $\{e_1,\ldots, e_{2n}\}$ of $\frg$ adapted to the Hermitian structure.  It is enough to prove the statement for $\alpha=e^k$, $k=1,\ldots, 2n$. In term of this basis, 
$$F^{n-1} = \lambda
\sum_{r=1}^n e^{1}\wedge\cdots \wedge \widehat{e^{2r-1}}\wedge \widehat{e^{2r}}\wedge \cdots \wedge
e^{2n},$$ 
where $\lambda$ is a positive real number and the symbol $\widehat{e^s}$ means that $e^s$
does not appear in the corresponding term.  Now, for any $k=1,\ldots, 2n$,
$$F^{n-1}\wedge de^k = F^{n-1} \wedge \sum_{i<j} c_{ij}^k e^{ij} = \lambda \sum_{r=1}^{n}
c_{2r-1,2r}^k\,e^{12\cdots n},$$ and it follows from \eqref{const-estr-balanced} that $F$ is
balanced if and only if the expression above vanishes. 
\end{proof}

\begin{remark}
{\rm
Condition~\eqref{cond_balanced2} is equivalent to the fact that $F$ is orthogonal to the image of
$d:\frg^* \to \alt ^2 \frg^*$, as stated in \cite{AGS, FG}.   }
\end{remark}

The image of $d:\frg^*\to \alt^2 \frg^*$ provides also a way to characterize abelian complex
structures.  Indeed, $J$  is abelian if and only if $J(d\alpha) = d\alpha,$ for any
$\alpha\in\frg^*$, i.e., $d\alpha$ is $J$-invariant.   
\medskip

From now on we will be interested in Hermitian structures combining abelian complex structures and
balanced metrics, which motivates the following

\begin{definition}
Let $(J,F)$ be a Hermitian structure on a Lie algebra $\frg$. If $J$ is abelian and $F$ is
balanced,  we will refer to the pair $(J,F)$ as an \emph{abelian balanced Hermitian structure}.  
\end{definition}

If we denote 
\begin{equation}\label{gamma}
\Gamma := \{\eta\in\alt^2\frg^*\,|\, J\eta = \eta \text{ and } F^{n-1}\wedge\eta = 0\},
\end{equation}
we obtain from previous results the following characterization of abelian balanced Hermitian
structures.

\begin{corollary}\label{im_gamma}
Let $\frg$ be a $2n$-dimensional unimodular Lie algebra and let $(J,F)$ be a Hermitian structure on
$\frg$. Then $(J,F)$ is abelian balanced if and only if $\text{im } (d:\frg^*\to \alt^2
\frg^*)\subseteq \Gamma$.
\end{corollary}

\begin{remark}\label{remark_balanced}
{\rm
In terms of an adapted basis $\{e_1,\ldots, e_{2n}\}$, a $2$-form
$\eta\in\alt^2\frg^*$ belongs to $\Gamma$ if and only if 
\begin{itemize}
\item[i)] $J\eta = \eta$: \quad  $\eta(e_{2r},e_{2s}) =
\eta(e_{2r-1},e_{2s-1}),\quad \eta(e_{2r},e_{2s-1}) = -\eta(e_{2r-1},e_{2s})$. 
\smallskip
\item[ii)] $F^{n-1}\wedge\eta = 0$:\quad $\sum_{r=1}^n
\eta(e_{2r-1},e_{2r})=0.$
\end{itemize}
}\end{remark}

It will be very useful to have an explicit description of the vector subspace $\Gamma\subseteq
\alt^2\frg^*$ in terms of a basis adapted to the Hermitian structure $(J,F)$.

\begin{lemma}
Let $(J,F)$ be an abelian balanced Hermitian structure on a $2n$-dimensional unimodular Lie algebra $\frg$ and let
$\{e_1,\ldots, e_{2n}\}$ be an adapted basis.  Then, a basis for $\Gamma$ is given by
$\mathcal B=\{\gamma_{ij}\,|\, 2i<j,\,\,  i=1,\ldots, n-1\}\cup \{\beta_{r}\mid r=1,\ldots, n-1\}$, where:
\begin{equation}\label{base_gamma}
\begin{cases}
\gamma_{ij}=e^{2i-1,j} - e^{2i}\wedge Je^{j},\\\
\beta_r =e^{2r-1, 2r}- e^{2r+1,2r+2}.
\end{cases}\end{equation}
In particular, $\text{dim }\Gamma = n^2-1$.
\end{lemma}

\begin{proof}
It is easy to see that the space $I$ of $J$-invariant $2$-forms is generated, in
terms of the dual of the adapted basis, by 
\begin{eqnarray*}
I&=&\langle\{e^{ij} + J(e^{ij})\,|\, i<j\}\rangle\\ 
&= &\langle \{e^{ij} + J(e^{ij})\,|\, i<j-1\} \cup \{e^{2r-1,2r}\,|\, r=1,\ldots, n\}\rangle\\
&=&\langle\{e^{2i-1,j} - e^{2i}\wedge Je^{j}\,|\, 2i<j,\,\, i=1,\ldots, n-1\} \cup \{e^{2r-1,2r}\,|\, r=1,\ldots, n\}\rangle.
\end{eqnarray*}
Therefore, we can express any $\eta\in\Gamma$ as $\eta = \displaystyle\sum_{2i<j} a_{ij}
\gamma_{ij} + \displaystyle\sum_{r=1}^n b_{r} e^{2r-1,2r}.$
Since $F^{n-1}\wedge\eta=0$, Remark~\ref{remark_balanced} implies that $\sum_{r=1}^n b_r=0$.  This
last condition allows us to express $\sum_{r=1}^n b_{r} e^{2r-1,2r}$ as $\sum_{j=1}^{n-1} m_{j}
\beta_j$, where $m_j = \sum_{r=1}^jb_r$.  This shows that $\mathcal B$ spans $\Gamma$ and since all
the elements of $\mathcal B$ are linearly independent, the proof is complete.
\end{proof}

%%%%%%%%%%%%%%%%%%%%%%%%%%%%%%%%%%%%%%%%%%
%%%%%%%%%%%%%%%%%%%%%%%%%%%%%%%%%%%%%%%%%%

\section{Curvature and holonomy of the Bismut connection}

Bismut proved in \cite{Bis} that given any Hermitian structure $(J,F)$ on a $2n$-dimensional manifold 
$M$ there is a unique connection preserving the Hermitian structure with
totally skew-symmetric torsion~$T$ given by
$g(X,T(Y,Z))=JdF(X,Y,Z)=-dF(JX,JY,JZ)$, $g$ being the associated metric. This torsion connection,
denoted by $\nabla$, is known as the {\it Bismut connection} of $(J,F)$ and it can be derived from
the Levi-Civita connection $\nabla^g$ of the Riemannian metric $g$ by $\nabla = \nabla^g
+\frac{1}{2}T$, where $T$ is identified with the 3-form~$JdF$.

According to \cite{FPS}, the holonomy group of the Bismut connection associated to any invariant
balanced $J$-Hermitian structure on a nilmanifold $M^{2n}$ is contained in $SU(n)$.

The aim of this section is to prove such a reduction of the holonomy of the Bismut connection in
the case of abelian balanced Hermitian structures on a unimodular, non necessarily nilpotent, Lie algebra
$\frg$. Moreover, if the center of $\frg$ is non-trivial, then the holonomy reduces to a proper
subgroup of $SU(n)$. First examples of this reduction can be found in~\cite{UV}. 

\medskip

By the well-known Ambrose-Singer theorem, the Lie algebra $\mathfrak{hol}(\nabla)$ of
the holonomy group of any linear connection $\nabla$ is
generated by the curvature endomorphisms of $\nabla$ together with their covariant derivatives. We
recall now the following objects that will help us to compute $\mathfrak{hol}(\nabla)$.

\medskip

Given any linear
connection $\nabla$ on an $m$-dimensional manifold, the connection 1-forms $\sigma^i_j$ are given
with respect to a local orthonormal frame $\{e_1,\ldots, e_m\}$  by
\begin{equation}\label{def-1-formas}
\sigma^i_j(e_k) = g(\nabla_{e_k}e_j,e_i),
\end{equation}
i.e. $\nabla_X e_j = \sigma^1_j(X)\, e_1 +\cdots+ \sigma^m_j(X)\,
e_m$. The curvature 2-forms $\Omega^i_j$ of $\nabla$ can be expressed in
terms of the connection 1-forms $\sigma^i_j$ by
\begin{equation}\label{curvature}
\Omega^i_j = d \sigma^i_j + \sum_{k=1}^{m}
\sigma^i_k\wedge\sigma^k_j,
\end{equation}
and the curvature endomorphisms $R(e_r,e_s)$ of the connection $\nabla$ are given in terms of the
curvature forms~$\Omega^i_j$ by $g(R(e_r,e_s)e_i,e_j) = - \Omega^i_j (e_r,e_s)$.
We will identify the curvature endomorphisms $R(e_r,e_s)$ with the 2-forms given by 
\begin{equation}\label{endo-curvature}
R^{rs}(e_i,e_j) = -\Omega^i_j(e_r, e_s).
\end{equation}
Finally, the covariant derivative $\nabla_{e_j} \gamma$ of any 2-form~$\gamma$
is given by
\begin{equation}\label{cov-deriv}
(\nabla_{e_j} \gamma) (e_r,e_s)= \sum_{k=1}^m\left( \sigma^k_s(e_j)\, \gamma(e_k,e_r) -
\sigma^k_r(e_j)\, \gamma(e_k,e_s) \right),
\quad\ j=1,\ldots, m.
\end{equation}

In the particular case of a left-invariant Hermitian structure $(J,F)$ on a $2n$-dimensional Lie
group, we can work at the Lie algebra level.  Choosing an adapted basis $\{e_1,\ldots, e_{2n}\}$ of
$\frg$, the Levi-Civita connection 1-forms $(\sigma^g)^i_j$ can be expressed in terms of the
structure constants $c_{ij}^k$ by
\[
(\sigma^g)^i_j(e_k) = -\frac12 \left( g(e_i,[e_j,e_k]) - g(e_k,[e_i,e_j]) +
g(e_j,[e_k,e_i]) \right)=\frac12(c^i_{jk}-c^k_{ij}+c^j_{ki}).
\]
Since the Bismut connection $\nabla$ is
given by $\nabla=\nabla^g + \frac12 T$, with torsion $T=JdF$, the Bismut connection 1-forms $\sigma^i_j$ are determined by
\begin{equation}\label{Bismut-1-forms}
\sigma^i_j(e_k)=(\sigma^g)^i_j(e_k) - \frac12 T(e_i,e_j,e_k) =
\frac12(c^i_{jk}-c^k_{ij}+c^j_{ki}) - \frac12 JdF(e_i,e_j,e_k).
\end{equation}

\medskip

Next we show that if the complex structure $J$ is abelian, then it is possible to simplify some of the expressions above. 

\begin{lemma}\label{lemma_torsion}
Let $\frg$ be a $2n$-dimensional Lie algebra and let $(J,F)$ be a Hermitian structure
on~$\frg$ where $J$ is abelian. If $\{e^1,\ldots, e^{2n}\}$ is an adapted basis, then
the torsion $3$-form of the Bismut connection is given by
%\begin{equation}\label{torsion}
\[ T=\sum_{i=1}^{2n} e^i\wedge de^i.\]
%\end{equation}
\end{lemma}

\begin{proof}
According to \eqref{adapted_basis}, $F$ can be expressed as $F=-\frac{1}{2}\sum_{i=1}^{2n} e^i\wedge Je^i$.  Then:
\[
-JdF = \frac12 Jd\left(\sum_{i=1}^{2n} e^i\wedge Je^i \right)  = \frac12 J\left(\sum_{i=1}^{2n} de^i\wedge Je^i - e^i\wedge
d(Je^i) \right).  
\]
Since $de^k$ are $J$-invariant, 
\[ -JdF= -\frac{1}{2} \left(\sum_{i=1}^{2n} e^i\wedge de^i + Je^i\wedge d(Je^i) \right)=  -\sum_{i=1}^{2n} e^i\wedge de^i. \]
\end{proof}

Using this lemma, the Bismut connection $1$-forms can be expressed in a simpler way.  
Observe that 
\begin{eqnarray*}
JdF(e_i,e_j,e_k)& =& \sum_{r=1}^{2n}e^r\wedge de^r (e_i,e_j,e_k)\\
&=&\sum_{r=1}^{2n}\left(e^r(e_i)de^r(e_j,e_k)- e^r(e_j)de^r(e_i,e_k)+ e^r(e_k)de^r(e_i,e_j)\right)\\
&=&(c_{jk}^i-c_{ik}^j+c_{ij}^k).
\end{eqnarray*}

\begin{lemma}\label{lema_1formas}
Under the conditions of Lemma~\ref{lemma_torsion}, the  Bismut connection $1$-forms
are given by
\begin{equation}\label{1-formas_conexion}
\sigma^i_j = -\sum_{k=1}^{2n} c_{ij}^k e^k,\quad 1\leq i<j\leq 2n.
\end{equation}
Moreover, the following conditions hold for $i, j=1,\ldots, n$:
\begin{equation}\label{1-formas_conexion2}
\sigma^{2i}_{2j}=\sigma^{2i-1}_{2j-1},\quad \sigma^{2i-1}_{2j} = -\sigma^{2i}_{2j-1}.
\end{equation}
\end{lemma}

\begin{proof}
The first equality follows directly from~\eqref{Bismut-1-forms}:
\[
2\sigma^i_j (e_k) = (c_{jk}^i - c_{ij}^k + c_{ki}^j)-(c_{jk}^i-c_{ik}^j+c_{ij}^k) = -2\,c_{ij}^k.
\]
For the second statement, it is enough to observe that since $de^k$ are $J$-invariant, 
 the structure constants satisfy 
\begin{equation}\label{condicion_cij}
c_{2i-1,2j-1}^{k} = c_{2i,2j}^{k} \text{ and  }c_{2i-1,2j}^{k} = -c_{2i,2j-1}^{k},
\end{equation} for any $i,j=1,\ldots, n$ and $k=1,\ldots, 2n$.
\end{proof}

\begin{corollary}\label{curv_invariante}
Under the conditions of Lemma~\ref{lemma_torsion}, the Bismut curvature $2$-forms $\Omega^i_j$ satisfy  
%\begin{equation}\label{2-formas_curvatura}
\[ \Omega^{2i}_{2j}=\Omega^{2i-1}_{2j-1},\quad \Omega^{2i-1}_{2j} = -\Omega^{2i}_{2j-1}. \]
%\end{equation}
\end{corollary}

\begin{proof}
Let us prove the first equality.  For the second one similar computations can be applied.
From the general expression of the curvature $2$-forms given in~\eqref{curvature} we have
\[ \Omega^{2i}_{2j} = d \sigma^{2i}_{2j}  + \sum_{k=1}^{2n}
\sigma^{2i}_{k} \wedge\sigma^{k}_{2j}.\]
Using~\eqref{1-formas_conexion2} we get $d \sigma^{2i}_{2j}=d \sigma^{2i-1}_{2j-1}$.  We split the
second summand according to the parity of $k$:
\begin{eqnarray*}
\sum_{k=1}^{2n}
\sigma^{2i}_{k} \wedge\sigma^{k}_{2j}&=&\sum_{r=1}^{n}
\sigma^{2i}_{2r} \wedge\sigma^{2r}_{2j}+\sum_{s=1}^{n}
\sigma^{2i}_{2s-1} \wedge\sigma^{2s-1}_{2j}\\
&=&\sum_{r=1}^{n}
\sigma^{2i-1}_{2r-1} \wedge\sigma^{2r-1}_{2j-1}+\sum_{s=1}^{n}
\sigma^{2i-1}_{2s} \wedge\sigma^{2s}_{2j-1}\\
&=&\sum_{k=1}^{2n}
\sigma^{2i-1}_{k} \wedge\sigma^{k}_{2j-1},
\end{eqnarray*}
where in the second equality we have used~\eqref{1-formas_conexion2} again. 
Therefore $\Omega^{2i}_{2j}=\Omega^{2i-1}_{2j-1}$.
\end{proof}

The previous results hold for any Hermitian structure $(J, F)$ on a Lie algebra $\frg$ such that
$J$ is abelian. In what follows we will apply these results for the particular case of an abelian
balanced Hermitian structure on unimodular Lie algebras.

\begin{lemma}\label{NR-gamma}
If $\eta\in\Gamma$, then $\nabla_X\eta\in\Gamma$, for any $X\in\frg$.
\end{lemma}

\begin{proof}
According to Remark~\ref{remark_balanced}, we have to verify that (i) $J\nabla_X\eta=\nabla_X\eta$, 
and (ii) $F^{n-1}\wedge\nabla_X\eta=0$, for
any $X\in\frg$. Consider an adapted basis $\{e^1,\ldots,e^{2n}\}$.  For case (i):
\begin{eqnarray*}
(J(\nabla_X\eta))(e_{2r},e_{2s})&=&(\nabla_X\eta)(e_{2r-1},e_{2s-1})\\
&=&\sum_{k=1}^{2n}\left( \sigma^k_{2s-1}(X)\, \gamma(e_k,e_{2r-1}) - \sigma^k_{2r-1}(X)\, \gamma(e_k,e_{2s-1}) \right)\\
&=&\sum_{p=1}^{n}\left( \sigma^{2p}_{2s-1}(X)\, \gamma(e_{2p},e_{2r-1}) - \sigma^{2p}_{2r-1}(X)\, \gamma(e_{2p},e_{2s-1}) \right)+\\
&&\sum_{p=1}^{n}\left( \sigma^{2p-1}_{2s-1}(X)\, \gamma(e_{2p-1},e_{2r-1}) - \sigma^{2p-1}_{2r-1}(X)\, \gamma(e_{2p-1},e_{2s-1}) \right)\\
&=&\sum_{p=1}^{n}\left( \sigma^{2p-1}_{2s}(X)\, \gamma(e_{2p-1},e_{2r}) - \sigma^{2p-1}_{2r}(X)\, \gamma(e_{2p-1},e_{2s}) \right)+\\
&&\sum_{p=1}^{n}\left( \sigma^{2p}_{2s}(X)\, \gamma(e_{2p},e_{2r}) - \sigma^{2p}_{2r}(X)\, \gamma(e_{2p},e_{2s}) \right)\\
&=&(\nabla_X\eta)(e_{2r},e_{2s}),
\end{eqnarray*}
where we have used~\eqref{1-formas_conexion2} and Remark~\ref{remark_balanced}. The remaining cases
can be proved in a similar way.

For case (ii), it is enough to prove the statement for $\eta\in\mathcal B$, the basis of $\Gamma$ given by~\eqref{base_gamma}. It is easy to see that $(\nabla_X\beta_j)(e_{2r-1},e_{2r}) = 0$ for any $r=1,\ldots, n$ and $j=1,\ldots, n-1$.  On the other hand, we consider $\gamma_{i, 2j-1} = e^{2i-1,2j-1} + e^{2i,2j}$ for $i<j$:
\begin{eqnarray*}
\displaystyle\sum_{r=1}^n\left(\nabla_X \gamma_{i, 2j-1}\right)
(e_{2r-1},e_{2r})&=&\displaystyle\sum_{r=1}^{n}\sum_{k=1}^{2n}\left( \sigma^k_{2r}(X)\, \gamma_{i,
2j-1}(e_k,e_{2r-1}) - \sigma^k_{2r-1}(X)\, \gamma_{i, 2j-1}(e_k,e_{2r}) \right)\\
&=&\sigma^{2i-1}_{2j}(X)- \sigma^{2j-1}_{2i}(X)=\sigma^{2i-1}_{2j}(X)+\sigma^{2i}_{2j-1}(X)\\[5pt]&=&0,
\end{eqnarray*}
where the last equality holds by~\eqref{1-formas_conexion2}.   The case for $\gamma_{i,2j}$  can be
treated in the same way.
\end{proof}

As a consequence of Lemma~\ref{NR-gamma} we obtain the following 

\begin{proposition}
Let $(J,F)$ be an abelian balanced Hermitian structure on a unimodular Lie algebra~$\frg$. Then,
the curvature endomorphisms of the Bismut connection $R^{rs}$ and their covariant derivatives of any
order  belong to $\Gamma$.  
\end{proposition}

\begin{proof}
According to Lemma~\ref{NR-gamma}  it suffices to prove that $R^{rs}\in\Gamma$. 
Let us see first that $R^{rs}$ is $J$-invariant for any $r,s=1,\ldots, 2n$, i.e. $J(R^{rs}) = R^{rs}$.   We compute 
\begin{eqnarray*}
J(R^{rs})(e_{2i},e_{2j})&=& R^{rs}(Je_{2i}, Je_{2j}) = R^{rs}(e_{2i-1}, e_{2j-1})\\
&=&-\Omega^{2i-1}_{2j-1}(e_r,e_s) = -\Omega^{2i}_{2j}(e_r,e_s) \\
&=&R^{rs}(e_{2i},e_{2j}),
\end{eqnarray*}
where we have used~\eqref{endo-curvature} and Corollary~\ref{curv_invariante}.  The remaining cases
can be checked analogously.

\medskip

Next we show that $F^{n-1}\wedge R^{rs}=0$ for any $r,s=1,\ldots, 2n$.  According to
Remark~\ref{remark_balanced}, it is enough to show that  $\sum_{k=1}^{n} R^{rs}(e_{2k-1},e_{2k})=0$.
This condition is equivalent to 
$\sum_{k=1}^{n} \Omega^{2k-1}_{2k}(e_{r},e_{s})=0,$ for any  $r,s$, that is to say $\sum_{k=1}^{n}
\Omega^{2k-1}_{2k}=0.$  By definition,
\begin{eqnarray*}
\sum_{k=1}^{n} \Omega^{2k-1}_{2k}&=& \sum_{k=1}^n \left(d\sigma^{2k-1}_{2k} + \sum_{r=1}^{2n}
\sigma^{2k-1}_{r}\wedge \sigma^r_{2k}\right) = d \left(\sum_{k=1}^n\sigma^{2k-1}_{2k} \right)+
\sum_{k,r}\sum_{s,t}c^s_{2k-1,r} c^t_{r,2k}e^{st}.
\end{eqnarray*}
Due to~\eqref{1-formas_conexion} and~\eqref{const-estr-balanced},
$\sum_{k=1}^n\sigma^{2k-1}_{2k}=0$ and therefore:
\begin{eqnarray*}
\sum_{k=1}^{n} \Omega^{2k-1}_{2k}&=& \sum_{k,r}\sum_{s,t}c^s_{2k-1,r} c^t_{r,2k}e^{st}
=\sum_{k,r}\sum_{s<t}(c^s_{2k-1,r} c^t_{r,2k} - c^t_{2k-1,r} c^s_{r,2k})e^{st}\\
&=&\sum_{s<t}\sum_{k,r}(c^s_{2k-1,r} c^t_{r,2k} - c^t_{2k-1,r} c^s_{r,2k})e^{st}.
\end{eqnarray*} 
The expression above is zero if and only if $$\sum_{k,r}(c^s_{2k-1,r} c^t_{r,2k} - c^t_{2k-1,r}
c^s_{r,2k})=0$$ for any $s<t$.  Now, depending on the parity of $r$, we can express the summand as
$$\sum_{k,m}(c^s_{2k-1,2m} c^t_{2m,2k} - c^t_{2k-1,2m} c^s_{2m,2k}) + \sum_{k,p}(c^s_{2k-1,2p-1}
c^t_{2p-1,2k} - c^t_{2k-1,2p-1} c^s_{2p-1,2k}).$$ 
Taking into account~\eqref{condicion_cij}, the previous sum transforms into
$$\sum_{k,m}(c^s_{2k-1,2m} c^t_{2m,2k} - c^t_{2k-1,2m} c^s_{2m,2k}) + \sum_{k,p}(c^s_{2k,2p}
c^t_{2p-1,2k} - c^t_{2k,2p} c^s_{2p-1,2k}),$$ 
which is zero after relabelling subindices in the second summand. 
\end{proof}

\bigskip

%%%%%%%%%%%%%%%%%%%%%%%%%%%%%%%%%%%%%%%%
%%%%%%%%%%%%%%%%%%%%%%%%%%%%%%%%%%%%%%%%%
\subsection{Role of the center of $\frg$}
%%%%%%%%%%%%%%%%%%%%%%%%%%%%%%%%%%%%%%%%%
%%%%%%%%%%%%%%%%%%%%%%%%%%%%%%%%%%%%%%%%%

The center $\mathfrak z$ of $\frg$ will play a key role in order to determine the holonomy group of
the Bismut connection $\nabla$ for abelian balanced Hermitian structures.  So, let us assume that
$\mathfrak z\neq 0$ and consider $\frg=\mathfrak z \oplus \mathfrak z^{\bot}$.  We recall that if
$J$ is an abelian complex structure, then $\mathfrak z$ is $J$-invariant and therefore also
$\mathfrak z^{\bot}$. 

\begin{proposition}
Let $(J,F)$ be a Hermitian structure on a Lie algebra $\frg$ with $J$ abelian. Then a $1$-form 
$\eta\in \frg^*$ satisfies $\nabla\eta=0$ if and only if $\eta$ is dual to an element of the 
center~$\frz$ of~$\frg$.
\end{proposition}

\begin{proof}
Any $\eta\in\frg^*$ can be written as $\eta(\cdot)=g(\cdot,Z_0)$ for a unique $Z_0\in\frg$. For any $X, Y\in\frg$ we compute
\begin{align*}
(\nabla_X \eta)(Y) & = -\eta(\nabla_XY)= -g(\nabla_XY, Z_0) \\
                   & = -g\left(\nabla^g_XY+\frac12 T(X,Y),Z_0\right) \\
									 & = -\frac12(g([X,Y],Z_0)-g([Y,Z_0],X)+g([Z_0,X],Y))+\frac12 dF(JZ_0,JX,JY) \\
									 & = -\frac12(g([X,Y],Z_0)-g([Y,Z_0],X)+g([Z_0,X],Y)) \\
									 & \qquad +\frac12 (g([JZ_0,JX],Y)+g([JX,JY],Z_0)+g([JY,JZ_0],X)  \\
									 & = g([Y,Z_0],X)\quad \text{since $J$ is abelian}.
\end{align*}
It follows immediately from this equation that $Z_0\in\frz$ if and only if $\nabla\eta=0$. 
\end{proof}

In order to study the curvature of the Bismut connection, we choose an adapted basis $\{e_1,\ldots,
e_{2n}\}$ of $\frg$ such that $\{e_1,\ldots, e_{2m}\}$ is a basis of $\mathfrak z^{\perp}$ and
$\{e_{2m+1},\ldots, e_{2n}\}$ is a basis of $\mathfrak z$.  Combining the previous result with
equations~\eqref{def-1-formas} and~\eqref{curvature} we obtain

\begin{corollary}
Let $(J,F)$ be a Hermitian structure on a Lie algebra $\frg$ with $J$ abelian. Then the Bismut 
connection $1$-forms and the curvature $2$-forms satisfy:
\begin{equation}\label{12-ceros}
\sigma^i_j = 0\quad\text{and }\quad \Omega^i_j=0,\quad\text{if }\,\,i>2m \text{ or } j>2m.
\end{equation}
\end{corollary}

In order to prove the reduction of the holonomy, let us define, for any $q=1,\ldots, n$,
\begin{equation}\label{gammaq}
\Gamma_q=\langle\mathcal B_q\rangle=\langle\{\gamma_{ij}\,|\, 2i<j<2q\}\cup \{\beta_{r}\mid r=1,\ldots, q-1\}\rangle.
\end{equation}
For the particular case of abelian balanced Hermitian structures, we obtain the following properties for the
curvature endomorphisms and their covariant derivatives:

\begin{lemma}\label{pertenenciaGm}
Let $(J,F)$ be an abelian balanced Hermitian structure on a unimodular Lie algebra~$\frg$ of dimension $2n$. If $\gamma\in\Gamma_m$, then $\nabla_X\gamma\in\Gamma_m$, for any $X\in \frg$, where $2m = \dim \frg -
\dim \mathfrak z$ as above.
\end{lemma}

\begin{proof}
By definition 
$$(\nabla_{X} \gamma) (e_r,e_s)= \sum_{k=1}^{2n}\left( \sigma^k_s(X)\, \gamma(e_k,e_r) -
\sigma^k_r(X)\, \gamma(e_k,e_s) \right).$$ 
Using~\eqref{12-ceros} it is possible to restrict the sum up to $k=2m$. Observe that if $j>2m$, then 
$\sigma_j^k(X)=\gamma(e_k,e_j) =0$ for any $k\leq 2m$, since $\gamma\in\Gamma_m$. We conclude that 
$(\nabla_{X} \gamma)(e_r,e_s)=0$ if $\max\{r,s\} > 2m$, that is to say, $\nabla_X \gamma\in\Gamma_m$ for any $X\in\frg$.   
\end{proof}

\begin{proposition}\label{curvatura_en_gammam}
Let $(J,F)$ be an abelian balanced Hermitian  structure on a unimodular Lie algebra~$\frg$ of dimension $2n$. The curvature endomorphisms $R^{rs}$ and their covariant derivatives of any order belong to $\Gamma_m$, where $2m = \dim \frg - \dim \mathfrak z$.
\end{proposition}

\begin{proof}
Taking into account Lemma~\ref{pertenenciaGm}, it suffices to prove that $R^{rs}\in\Gamma_m$ for
any $r,s$. Observe first that $R^{rs}(e_i,e_j) = -\Omega^i_j(e_r,e_s) =0$ if $\max\{i,j\}>2m$. 
Since $R^{rs}\in\Gamma$, let us express 
$$R^{rs} = \sum_{2i<j} a_{ij}^{rs}\gamma_{ij} + \sum_{p=1}^{n-1} b_p^{rs}\beta_p.$$
If $2i<j$ and $j>2m$ then $0=R^{rs}(e_i,e_j)=a_{ij}^{rs}$.  On the other hand,
$$R^{rs}(e_{2k-1},e_{2k})=\begin{cases}\begin{array}{ll}b_1^{rs},& \text{if  }\, k=1,\\
b_k^{rs}-b_{k-1}^{rs},& \text{if  }\, 2\leq k\leq n-1,\\
-b_{n-1}^{rs},& \text{if  }\, k=n.
 \end{array}\end{cases}$$  
Since $R^{rs}(e_{2k-1},e_{2k})=0$ if $k\geq m+1$, we obtain that $b_k^{rs}=0$ for $m+1\leq k\leq
n-1$.  Therefore, $R^{rs}\in\Gamma_m$.
\end{proof}

\bigskip

%%%%%%%%%%%%%%%%%%%%%%%%%%%%%%%%%%%%%%%%%%

%%%%%%%%%%%%%%%%%%%%%%%%%%%%%%%%%%%%%%%%%%
\subsection{Holonomy of the Bismut connection}
%%%%%%%%%%%%%%%%%%%%%%%%%%%%%%%%%%%%%%%%

According to \cite{FPS}, the holonomy group of the Bismut connection associated to any invariant
balanced $J$-Hermitian structure on a nilmanifold $M^{2n}$ is contained in $SU(n)$.

Our aim in what follows is to prove such a reduction of the holonomy of the Bismut connection in
the case of abelian balanced Hermitian structures on a unimodular, non necessarily nilpotent, Lie
algebra~$\frg$. Moreover, if the center of $\frg$ is non-trivial, then the holonomy reduces to a
proper subgroup of $SU(n)$. 

To prove the reduction of the holonomy, we recall the natural identification $\phi$ of the space
of $2$-forms $\alt^2\frg^*$ with the Lie algebra $\mathfrak{so}(m)$, $\phi:\alt^2\frg^*\to
\mathfrak{so}(m) $, where $\dim \frg = m$.  Indeed, fixing a basis $\{e_1,\ldots, e_{m}\}$ of
$\frg$, the $2$-form $e^{ij}$ is identified with the skew-symmetric matrix $E_{ij}-E_{ji}$, where
$E_{ij}$ is the matrix whose entries are all zero except for the element $(i,j)$ which is equal 
to~$1$,  i.e. $\phi(e^{ij}) = E_{ij} - E_{ji}$.

\begin{lemma}\label{phi}
Let $(J,F)$ be an abelian balanced Hermitian  structure on a unimodular $2n$-dimensional Lie
algebra~$\frg$ and $\Gamma\subseteq\alt^2\frg^*$ given by~\eqref{gamma}.  Then,
$\phi(\Gamma)\subseteq \mathfrak{so}(2n)$ is a Lie subalgebra isomorphic to $\mathfrak{su}(n)$. 
Moreover, $\phi(\Gamma_q)\subseteq\phi(\Gamma)$ is a Lie subalgebra isomorphic to
$\mathfrak{su}(q)\subseteq \mathfrak{su}(n)$, where $\Gamma_q$ is defined by~\eqref{gammaq}.
\end{lemma}

\begin{proof}
First we describe a particular embedding  $\mathfrak{su}(n)\hookrightarrow \mathfrak{so}(2n)$.
The Lie algebra $\mathfrak{su}(n)$ consists of $n\times n$ complex matrices that are traceless and
antihermitian. Fix the standard complex structure $J_0$ on $\R^{2n}$ given by $J_0e_{2k-1}=-e_{2k}$,
$k=1,\ldots,n$. Then $\mathfrak{su}(n)$ can be identified with a subalgebra of
$\mathfrak{gl}(n,\C):=\{M\in\mathfrak{gl}(2n,\R)\mid MJ_0=J_0M\}$. More precisely, a matrix in 
$\mathfrak{su}(n)$ can be considered as a $2n\times 2n$ real skew-symmetric matrix of the form 
\[ M= \begin{pmatrix}
      A_{11} & \cdots& A_{1n} \\
      \vdots & & \vdots \\
      A_{n1} & \cdots& A_{nn}  
      \end{pmatrix},
\]
where each $A_{ij}$ is a $2\times 2$ block of the form
\[ A_{ij}=\begin{pmatrix}
           a_{ij} & b_{ij} \\
           -b_{ij} & a_{ij}
          \end{pmatrix} \text{ and } A_{ji}=-A_{ij}^t \text{ if } i < j,\quad 
  A_{ii}=\begin{pmatrix}
           0 & c_{i} \\
           -c_{i} & 0
          \end{pmatrix},\,\,\text{with  } \sum_{i=1}^n c_{i}=0.
\]
If $M=(m_{ij}),$ for $i,j=1,\ldots, 2n$, then $a_{ij} = m_{2i-1,2j-1}=m_{2i,2j}$ and 
$b_{ij} = m_{2i-1,2j}=-m_{2i,2j-1}$, while $c_i=m_{2i-1,2i}=-m_{2i,2i-1}$.
Now, taking into account that $\sum_{i=1}^n c_{i}=0$, it is immediate to see that
$\{\phi(\gamma_{ij})\mid 2i<j\} \cup \{\phi(\beta_r)\mid r=1,\ldots,n-1\}$ is a basis of
$\mathfrak{su}(n)$,  where  
$\gamma_{ij}$ and $\beta_r$ are defined by~\eqref{base_gamma}.
Therefore $\phi(\Gamma)=\mathfrak{su}(n)$ as described above. 

For the last statement observe that $\phi(\Gamma_q )$ is the subset of $\mathfrak{su}(n)$
consisting of matrices whose last $2n-2q$ rows and last $2n-2q$ columns are zero, which is
canonically isomorphic to $\mathfrak{su}(q)$.
\end{proof}

\begin{example}
{\rm
For instance, if $n=4$, a basis for $\Gamma$, or equivalently for $\mathfrak{su}(4)$ via $\phi$, is given by:
\begin{equation*}\label{base_su4}
\begin{array}{ll}
\langle\!\!\!&\!\!\!e^{12}-e^{34},\, e^{13}+e^{24},\, e^{14}-e^{23},\\[3pt]
&\!\!\!e^{34}-e^{56},\, e^{15}+e^{26},\, e^{16}-e^{25},\, e^{35}+e^{46},\, e^{36}-e^{45},\\[3pt]
&\!\!\!e^{56}-e^{78},\, e^{17}+e^{28},\, e^{18}-e^{27},\, e^{37}+e^{48},\, e^{38}-e^{47},\, e^{57}+e^{68},\, e^{58}-e^{67}\rangle,
\end{array}
\end{equation*}
where elements in the first row generate $\mathfrak{su}(2)$ and the first and the second ones generate $\mathfrak{su}(3)$. 
}\end{example}

The next theorem states the reduction of the holonomy for abelian balanced Hermitian
structures:

\begin{theorem}\label{holonomia}
Let $\frg$ be a unimodular Lie algebra of dimension $2n$ endowed with an abelian balanced 
Hermitian structure $(J,F)$ and consider the
associated Bismut connection~$\nabla$.  If $\dim \mathfrak z=2k$, then $\mathfrak{hol}(\nabla)\subseteq \mathfrak {su}(n-k)$.
\end{theorem}

\begin{proof}
According to the well-known Ambrose-Singer theorem, the Lie algebra of the (restricted) holonomy
group of $\nabla$, $\mathfrak{hol}(\nabla)$ is generated as a vector space by the curvature
endomorphisms $R^{rs}$ and their covariant derivatives of any order.  By
Proposition~\ref{curvatura_en_gammam} all these endomorphisms belong to $\Gamma_m$ which can be
identified with $\mathfrak{su}(m)$, where $m=n-k$ (see Lemma~\ref{phi}).
\end{proof}

%%%%%%%%%%%%%%%%%%%%%%%%%%%%%%%%%%%%%%%%%%
%%%%%%%%%%%%%%%%%%%%%%%%%%%%%%%%%%%%%%%%%%

\section{Abelian balanced Hermitian structures on nilpotent Lie algebras}

In this section we consider the particular case of abelian balanced Hermitian structures $(J,F)$ on
$2n$-dimensional nilpotent Lie algebras. 

We consider first the more general case of a Hermitian structure $(J,F)$ with abelian $J$ on a Lie algebra $\frg$ with $\dim\frg=2n$. According to \cite{S}, due to the fact that $J$ is
abelian, there exists a basis of $(1,0)$-forms $\{\alpha^1,\ldots, \alpha^n\}$ such that 
\begin{equation}\label{salamon}
d\alpha^j\in\alt^2\langle\alpha^1,\ldots, \alpha^{j-1}, \bar{\alpha}^1,\ldots \bar{\alpha}^{j-1}\rangle,\quad j=1,\ldots, n.
\end{equation}

We extend naturally the inner product $g(\cdot,\cdot)=F(J\cdot,\cdot)$ to $\frg^{1,0}$. Applying 
the Gram-Schmidt process to the basis $\{\alpha^1,\ldots, \alpha^n\}$, we obtain an orthogonal basis 
$\{\omega^1,\ldots, \omega^n\}$ of $\frg^{1,0}$ such that 
$\text{span}_\C\{\omega^1,\ldots,\omega^j\}=\text{span}_\C\{\alpha^1,\ldots,\alpha^j\}$ for any $j$. Therefore the new orthogonal basis also satisfies the structural equations \eqref{salamon}. By rescaling 
appropriately and setting $\omega^j=e^{2j-1} + i\,e^{2j}$, we obtain a real orthonormal basis 
$\{e^1,\ldots, e^{2n}\}$ of $\frg$ such that $Je^{2j-1}=-e^{2j}$, that is, an adapted basis for the 
abelian Hermitian structure. 

A consequence of \eqref{salamon} is that the first Betti number of $\frg$ is greater or equal than
2. Next we show that this lower bound increases for the case of nilpotent Lie algebras endowed with
abelian balanced Hermitian structures.

\begin{proposition}\label{prop_top} 
Let $(J,F)$ be an abelian balanced Hermitian structure on a nilpotent Lie algebra $\frg$ of
dimension $2n$. Then $b_1(\frg)\geq 4$.
\end{proposition}

\begin{proof}
According to~\eqref{salamon}, there exists a basis of $(1,0)$-forms $\{\omega^1,\ldots, \omega^n\}$
such that $d\omega^1=0$ and $d\omega^2 = a\,\omega^1\wedge \bar\omega^{1}$. Taking real
and imaginary parts, we obtain that $de^1=de^2=0$ and  $de^3=2se^{12}, \,
de^4=-2re^{12}$, where $a=r+is, \, r,s\in\R$. Since the Hermitian structure $(J,F)$ is abelian
balanced and $\frg$ is unimodular, we have that $de^3,\, de^4\in\Gamma =  \langle e^{12}-e^{34},\, e^{13}+e^{24},\, e^{14}-e^{23}\rangle$ but this happens if and only
if $r=s=0$ and therefore $b_1(\frg)\geq 4$. 
\end{proof}

If $\dim \frg=4$, the previous proposition says that the only nilpotent Lie algebra admitting an
abelian balanced Hermitian structure is the abelian one (observe that in dimension four, balanced
metrics are K\"{a}hler, and therefore, the Lie algebra must be abelian~\cite{BG}).  For $\dim
\frg=6$, these structures are studied in \cite{UV}.

\

In what follows we denote by $\frh_{2n+1}$ the Lie algebra of the Heisenberg group of dimension
$2n+1$.  

\begin{proposition}\label{heisenberg}
The Lie algebras $\frg\cong\frh_{2n+1}\times
\R^{2k+1}$, where $\frh_{2n+1}$ is a real Heisenberg group, admit abelian balanced Hermitian 
structures $(J,F)$ if and only if $n\geq 2$.
\end{proposition}

\begin{proof}
According to~\cite{R}, any complex structure in $\frh_{2n+1}\times\R^{2k+1}$ is abelian and it is equivalent to one of the form 
\begin{equation}\label{complex-str-heis}
J e_{2i-1}=\pm e_{2i},\quad 1\leq i\leq n, \quad  J z_{2j} = -z_{2j+1}, \quad 0\leq j\leq k,
\end{equation}
$\{e_1,\ldots, e_{2n},\,z_0,\ldots, z_{2k+1}\}$ is a basis of $\frh_{2n+1}\times\R^{2k+1}$ satisfying $[e_{2i-1}, e_{2i}] = z_0$ for any $i=1,\ldots,n$.   Moreover,  according to~\cite{ABD}, there exist $[n/2]+1$ 
equivalence classes of complex structures, depending on the number of minus signs in the complex 
structure \eqref{complex-str-heis}. Concretely, consider the following complex structures:
$$J_0 e_{2i-1} = e_{2i},\quad  i = 1,\ldots, n,$$
$$\begin{cases}\begin{array}{ll}J_r e_{2i-1} = -e_{2i},\,\,  &i=1,\ldots, r,\\ J_r e_{2i-1} = e_{2i},\,\, &i =r+1,\ldots, 
n,\end{array}\end{cases}\quad r=1,\ldots, n-1,$$
$$J_n e_{2i-1} = -e_{2i},\quad  i = 1,\ldots, n.$$
Now, it is clear that $J_{k}$ is equivalent to $J_{n-k}$ and therefore, we can consider $0\leq r\leq [n/2]$.  

If $n=1$, then the only complex structure up to equivalence is $J_0$ and~\eqref{corchetes_balanced} cannot be satisfied, which means that there are no balanced metrics.  On the other hand, if $n\geq 2$, let us consider structures $J_r$ with $r\geq 1$.  Clearly, 
the basis $\{e_1,\ldots, e_{2n},\,z_0,\ldots, z_{2k+1}\}$ is not adapted to the complex structure 
$J_r$.  However, we  can define a new adapted basis $\{f_1,\ldots, 
f_{2n+1},\,z_0,\ldots, z_{2k+1}\}$ in the following way:
\[ f_i = \sqrt{n+1-2r}\,\,e_{i},\,\text{ for } i=1,2,\quad  f_{2i} = -e_{2i},\, \text{ for } i = 
r+1,\ldots, n,\] 
and $f_i = e_i$ for the remaining cases.  Moreover,
\begin{eqnarray}\label{corchetes-heis}
\nonumber[f_1, f_2] &= &(n+1-2r) z_0,\\[5pt] 
\nonumber[f_{2i-1}, f_{2i}] &=&z_0,\, \text{ for } i = 1,\ldots, r,\\[5pt]   
\nonumber[f_{2i-1}, f_{2i}] &=& -z_0,\, \text{ for } i=r+1,\ldots, n.
\end{eqnarray}
Now,
\begin{eqnarray*}
\sum_{i=1}^n [J_rf_{2i-1}, f_{2i-1}] &=& [f_1, f_2] + \sum_{i=2}^r[f_{2i-1}, f_{2i}] + 
\sum_{i=r+1}^n[f_{2i-1}, f_{2i}]\\[5pt] &=& \left[(n+1-2r) + (r-1) - (n-r)\right] z_0 = 0.
\end{eqnarray*}
By Lemma~\ref{J-bon}, the metric that makes  $\{f_1, \ldots, f_{2n+1},\, z_0,\ldots, z_{2k+1}\}$ orthonormal is balanced.
\end{proof} 

Moreover, we obtain the following result:

\begin{proposition}\label{Heis}
The Lie algebra $\frh_{2n+1}\times\R^{2k+1}$ possesses $[\frac n2]+1$ equivalence classes of 
abelian complex structures. If $n\geq 2$, among them, only $[\frac n2]$ admit balanced metrics.
\end{proposition}

\begin{proof}
In the proof of the previous proposition we have seen that $J_r$, $r\neq 0$, admit balanced 
metrics.  Let us see now that in fact $J_0$ does not admit this type of metrics.  The complex structure $J_0$ is equivalent to $J_n$, therefore  the basis $\{e_1,\ldots, e_{2n}, z_0,\ldots, z_{2k+1}\}$ is $J_n$-adapted and $[e_{2i-1}, e_{2i}] = z_0$ for $i=1,\ldots, n$.  In particular, this basis does not satisfy equation~\eqref{corchetes_balanced}. Let us see that no $J_n$-adapted basis may satisfy that equation.

Consider a generic basis of the form $\{f_1,\ldots, f_{2n}, z_0,\ldots, z_{2k+1}\}$ where
$
f_i = \sum_{j=1}^{2n} \lambda^{i}_{j} e_{j},$ for $ i=1,\ldots, 2n,  %\sum_{j=1}^{n} \lambda^{k}_{2j-1} e^{2j-1} +   \sum_{j=1}^{n} \lambda^{k}_{2j} e^{2j},
$
and $(\lambda^i_j)\in GL(2n, \R)$.
Imposing the condition that this basis is $J_n$-adapted, that is $J_n f_{2i-1} = -f_{2i}$, we obtain the conditions 
%\begin{equation}\label{cambio_base_heis}
\[ \lambda^{2i-1}_{2j-1} = \lambda^{2i}_{2j},\quad \lambda^{2i-1}_{2j} = -\lambda^{2i}_{2j-1}.\]
%\end{equation}
Now, 
\begin{eqnarray*}
[f_{2i-1}, f_{2i}]&=& \left[\sum_{j=1}^{n} \lambda^{2i-1}_{2j-1} e_{2j-1} +   \sum_{j=1}^{n} 
\lambda^{2i-1}_{2j} e_{2j},\,\, \sum_{j=1}^{n} \lambda^{2j}_{2j-1} e_{2j-1} +   \sum_{j=1}^{n} 
\lambda^{2i}_{2j} e_{2j}\right]\\[5pt]
&=&\left[\sum_{j=1}^{n} \lambda^{2i-1}_{2j-1} e_{2j-1},\,\, \sum_{j=1}^{n} \lambda^{2i}_{2j} 
e_{2j}\right]  +  \left[ \sum_{j=1}^{n} \lambda^{2i-1}_{2j} e_{2j},\,\, \sum_{j=1}^{n} 
\lambda^{2j}_{2j-1} e_{2j-1}\right]\\[5pt]
&=&\sum_{j=1}^{n} \left(\lambda^{2i-1}_{2j-1} 
\lambda^{2i}_{2j}-\lambda^{2i}_{2j-1}\lambda^{2i-1}_{2j}\right) [e_{2j-1}, e_{2j}]\\[5pt]
&=&\left(\sum_{j=1}^{n} \left(\lambda^{2i-1}_{2j-1}\right)^2 +  
\left(\lambda^{2i-1}_{2j}\right)^2\right) z_0,
\end{eqnarray*}
i.e., $[f_{2i-1}, f_{2i}] = \mu_i z_0$ where $\mu_i>0$.  Now it is clear 
that~\eqref{corchetes_balanced} cannot be satisfied. 
\end{proof}

\medskip

\begin{remark}
{\rm According to \cite{CFL}, the Lie algebra $\frh_{2n+1}\times \R$, $n\geq 
1$, admits a $J_0$-Hermitian metric $g$ which is locally conformally K\"ahler, i.e. the fundamental 
form $F$ satisfies $dF=\theta\wedge F$ for some closed $1$-form $\theta$. In this case, we have 
$d^*F\neq 0$. Moreover, it was proved in \cite{AO} that these are the only Lie algebras admitting locally conformally K\"ahler structures with abelian complex structure.}
\end{remark}

As a consequence of previous results we can conclude that abelian balanced Hemitian structures also impose conditions on the center of the nilpotent Lie algebra and therefore on the holonomy group of the associated Bismut connection.

\begin{proposition}\label{prop_top2} 
Let $(J,F)$ be an abelian balanced Hermitian structure on a nilpotent Lie algebra $\frg$ of
dimension $2n\geq 6$. Then $2\leq \dim \mathfrak z \leq 2n-4$. 
\end{proposition}

\begin{proof} 
It is known that any nilpotent Lie algebra has non-trivial center. Moreover, since $J$ is
abelian, then the center is $J$-invariant and therefore $\dim \mathfrak z\geq 2$.
%The abelianity of $J$ implies that $\{e_{2n-1}, e_{2n}\}\subset \mathfrak z$. 
Let us suppose now that $\dim\frz=2n-2$.  This means that $\frg$ is isomorphic to $\frh_3\times
\R^{2n-3}$.  By Proposition~\ref{heisenberg}, this algebra does not admit balanced Hermitian metrics.
\end{proof}

\begin{corollary}\label{reduction}
Let $(J,F)$ be an abelian balanced Hermitian structure on a nilpotent Lie algebra $\frg$ of
dimension $2n\geq 6$. Then, the holonomy group of the Bismut connection reduces to a proper subgroup of
$SU(n)$.
\end{corollary}

\begin{proof}
The result follows directly from Theorem~\ref{holonomia} and the fact that $\dim \mathfrak z\geq
2$.
\end{proof}

The abelian balanced Hermitian structures on $\frh_{2n+1}\times\R^{2k+1}$ provide the strongest reduction of the holonomy of the Bismut connection.  
\begin{proposition}
Let $\nabla$  be the Bismut connection associated to any abelian balanced Hermitian structure $(J, F)$ on
$\frh_{2n+1}\times\R^{2k+1}$, $n\geq 2$.  Then,  $\mathfrak{hol}(\nabla)=\mathfrak {u}(1)$.
\end{proposition}

\begin{proof}
Consider an adapted basis $\{e_1, \ldots, e_{2n},\, z_1,\ldots, z_{2k+1},\, z_0\}$ of 
$\frh_{2n+1}\times  \R^{2k+1}$ and its dual $\{e^1, \ldots, e^{2n+2k+2}\}$. The structure equations 
for any abelian balanced Hermitian structure are
given by $$de^i =0,\quad i=1,\ldots, 2n+2k+1,\quad de^{2n+2k+2}=\sum_{r=1}^{n-1} b_r \beta_r,$$ 
with $\beta_r$ defined in \eqref{base_gamma}. Since the
only non-zero structure constants are $c_{2i-1, 2i}^{2n+2k+2}$  for $i=1,\ldots, n$, the only
non-zero Bismut connection $1$-forms are $\sigma^{2i-1}_{2i}$, which are multiples of $e^{2n+2k+2}$ by
\eqref{1-formas_conexion}.  Moreover,  the curvature $2$-forms $\Omega^{2i-1}_{2i}$ are multiples of
$de^{2n+2k+2}$ and the curvature endomorphisms $R^{2i-1,2i}$ can be identified also with multiples of
$de^{2n+2k+2}$ via \eqref{endo-curvature}. In this way, the space generated by the curvature
endomorphisms is 1-dimensional.

For covariant derivatives, observe that if $X$ is not a multiple of $z_0$, then $\sigma^p_q(X) =0$
and therefore $\nabla_{X}\gamma=0$. Now, if $X=z_0$, consider $\gamma=e^{2i-1,2i}$. The terms
$\sigma_s^k(z_0)\gamma(e_k,e_r)$ in ~\eqref{cov-deriv} are non-zero if and only if $k=2i-1,\,
r=s=2i$ or $k=2i,\, r=s=2i-1$  but in both cases, $\nabla_{z_0}\gamma(e_r,e_r)=0$. Therefore,
$\nabla_{X}R=0$ for any curvature endomorphism $R$, which means that $\mathfrak{hol}(\nabla)=\langle
de^{2n+2k+2}\rangle$ and in particular, $\dim\mathfrak{hol}(\nabla)=1$.
\end{proof}

\begin{remark}\label{nil}{\rm
The result of Corollary~\ref{reduction}  does not hold in general for solvable Lie algebras.  For instance, as it was
shown in \cite{UV}, the $6$-dimensional Lie algebra with structure equations
$$de^1=de^2=0,\quad de^3 = -e^{13}-e^{24},\quad de^4 = -e^{14}+e^{23},\quad de^5 = e^{15}+e^{26},\quad de^6 = e^{16}-e^{25},$$
 endowed
with the abelian balanced Hermitian structure $(J, F)$, where $Je^{2i-1} = -e^{2i}$ and
$F=e^{12}+e^{34}+e^{56}$, has Bismut holonomy group equal to $SU(3)$.  According to \cite{ABD}, this
is the only $6$-dimensional unimodular non-nilpotent Lie algebra that admits abelian complex
structures.}
\end{remark}

%%%%%%%%%%%%%%%%%%%%%%%%%%%%%%%%%%%%%%%%%%%%%%%%%%%%%%%

\subsection{Abelian balanced Hermitian structures on 8-dimensional nilpotent Lie algebras}

Let $\frg$ be an $8$-dimensional nilpotent Lie algebra endowed with an abelian balanced Hermitian 
structure $(J,F)$.  Using Corollary \ref{im_gamma}, Propositions \ref{prop_top} and \ref{prop_top2}, 
we can consider an adapted basis $\{e_1,\ldots,e_8\}$  of $\frg$ such that the Lie bracket is 
generically
given by:
$$
\begin{array}{ll}
[e_1,e_3]=[e_2,e_4]=-c_{13}^5e_5-c_{13}^6e_6-c_{13}^7e_7-c_{13}^8e_8,&\quad\!
[e_1,e_2]=-c_{12}^5e_5-c_{12}^6e_6-c_{12}^7e_7-c_{12}^8e_8,\\[5pt]

[e_1,e_4]=-[e_2,e_3]=-c_{14}^5e_5-c_{14}^6e_6-c_{14}^7e_7-c_{14}^8e_8,&\quad\!
[e_3,e_4]=c_{12}^5e_5+c_{12}^6e_6-c_{34}^7e_7-c_{34}^8e_8,\\[5pt]

[e_1,e_5]=[e_2,e_6]=-c_{15}^7e_7-c_{15}^8e_8,&\quad\!
[e_5,e_6]=(c_{12}^7+c_{34}^7)e_7+(c_{12}^8+c_{34}^8)e_8,\\[5pt]

[e_1,e_6]=-[e_2,e_5]=-c_{16}^7e_7-c_{16}^8e_8,&\\[5pt]
[e_3,e_5]=[e_4,e_6]=-c_{35}^7e_7-c_{35}^8e_8,&\\[5pt]
[e_3,e_6]=-[e_4,e_5]=-c_{36}^7e_7-c_{36}^8e_8,&
\end{array}
$$
where coefficients $c_{ij}^k\in\R$ must satisfy relations that ensure the Jacobi identity. 
Analogously, the corresponding structure equations are:
\begin{equation}\label{estructure_equations_8}
\begin{cases}
\begin{array}{l}
de^i=0,\quad i=1,2,3,4,\\
de^5=c_{12}^5\beta_{1} +c_{13}^5\g_{13} +c_{14}^5\g_{14},\\
de^6=c_{12}^6\beta_{1} +c_{13}^6\g_{13} +c_{14}^6\g_{14},\\
de^7=c_{12}^7\beta_{1} +(c_{12}^7+c_{34}^7)\beta_{2}+c_{13}^7\g_{13}+c_{14}^7\g_{14}+c_{15}^7\g_{15}+c_{16}^7\g_{16} +c_{35}^7\g_{25} +c_{36}^7\g_{26},\\
de^8=c_{12}^8\beta_{1} +(c_{12}^8+c_{34}^8)\beta_{2}+c_{13}^8\g_{13}+c_{14}^8\g_{14}+c_{15}^8\g_{15}+c_{16}^8\g_{16}+c_{35}^8\g_{25} +c_{36}^8\g_{26},
\end{array}
\end{cases}
\end{equation}
where $\beta_r$ and $\g_{ij}$ are defined in~\eqref{base_gamma}.  Observe that $e_7$ and $e_8$ are
always central elements and $e_5$ and $e_6$ are central if and only if $c_{12}^i+c_{34}^i =
c_{15}^i=c_{16}^i=c_{35}^i=c_{36}^i=0$ for $i=7$ and $8$.

\begin{proposition}\label{4-dim-center}
Let $\frg$ be an $8$-dimensional nilpotent Lie algebra with $4$-dimensional center endowed with an
abelian balanced Hermitian structure $(J,F)$.  Then, $\frg$ is isomorphic to one and only one algebra in the
following list:
\begin{enumerate}
\item[I.1)] $\frg_1$: \quad $de^i=0,\quad  i=1,2,3, 4,5,\quad  de^6= e^{12}-e^{34},\quad  de^7= e^{13}+e^{24},\quad  de^8= e^{14}-e^{23}$;
\item[I.2)] $\frg_2$: \quad $de^i=0,\quad  i=1,2,3, 4,5, 6,\quad   de^7= e^{13}+e^{24},\quad  de^8= e^{14}-e^{23}$;
\item[I.3)] $\frg_3$: \quad $de^i=0,\quad  i=1,2,3, 4,5,6,7,\quad    de^8= e^{12}-e^{34}.$
\end{enumerate}
\end{proposition}

\begin{proof}
Let $\{e_1,\ldots, e_8\}$ be an adapted basis such that the center of $\frg$ is given by $\mathfrak
z = \langle e_5, e_6, e_7, e_8\rangle$, i.e. the structure equations are of type: $$de^i=0,\quad
i=1,2,3,4,\quad de^i=c_{12}^i\beta_{1} +c_{13}^i\g_{13} +c_{14}^i\g_{14}, \quad i=5,6,7,8.$$ 
Clearly, the Jacobi identity is satisfied for any value of $c_{ij}^k$.  Since $de^i\in\langle
\beta_1,\, \g_{13},\, \g_{14}\rangle$ for $i=5,6,7,8$, we can assume that $de^5=0$.  Now, depending
on the dimension of  $\langle de^6, de^7, de^8\rangle$, it is not hard to show that the Lie algebra
must be isomorphic to one of these:
\begin{equation*}
\begin{array}{lllll}
de^i=0,\quad  i=1,2,3, 4,5,\quad & de^6= \beta_{1},\quad & de^7= \g_{13},\quad & de^8= \g_{14},\\
de^i=0,\quad  i=1,2,3, 4,5,\quad & de^6=0,\quad & de^7= \g_{13},\quad & de^8= \g_{14},\\
de^i=0,\quad  i=1,2,3, 4,5,\quad & de^6= 0,\quad & de^7=0,\quad & de^8= \beta_{1}.
\end{array}
\end{equation*}
\end{proof}

\begin{remark}{\rm 
The Lie algebras that appear in the previous proposition are algebras associated to groups of type
$H$: $\frg_1$ is the product of the real Lie algebra of the quaternionic Heisenberg group and $\R$;
$\frg_2$ is the product of the real Lie algebra of the complex Heisenberg group and $\R^2$ and
$\frg_3$ is $\frh_5\times \R^3$.  These algebras are the only nilpotent ones that admit abelian
hypercomplex structures in dimension 8, \cite{DF}.  
}\end{remark}

\medskip

\begin{proposition} \label{prop:center4_dim8} 
Let $\frg$ be an $8$-dimensional nilpotent Lie algebra with $2$-dimensional center endowed with an
abelian balanced Hermitian structure $(J,F)$.  Then, $\frg$ satisfies one of the following conditions:
\begin{enumerate}
\item[II.1)] $de^5 = de^6=0$ and $\frg$ is 2-step.
\item[II.2)] $de^5$ and $de^6$ are linearly independent and $\frg$ is 3-step.
\end{enumerate}
\end{proposition}

\begin{proof}
For case II.1) observe that $de^5 = de^6 =0$  is equivalent to say that $e_5, e_6\notin [\frg, 
\frg]$. In fact, $[\frg, \frg]\subseteq\langle e_7, e_8\rangle$ and therefore 
$[[\frg,\frg],\frg]=0$.

For case II.2), let $\{e_1',\ldots, e_8'\}$ be an adapted basis such that $\mathfrak z = \langle 
e_7', e_8'\rangle$,
i.e. the structure equations are given by \eqref{estructure_equations_8}, where some of the
coefficients $c_{12}^i+c_{34}^i,\,c_{15}^i,\,c_{16}^i,\,c_{35}^i,\,c_{36}^i$ for $i=7$ or $i=8$ are
non-zero and suppose that $de'^5$ and $de'^6$ are linearly dependent, i.e., there exists 
$\lambda\in\mathbb R$ such that $de'^5 = \lambda de'^6$.  We can define a new basis in the following 
way: $$e^i = e'^i,\, i\neq 5,6,\quad e^5 = e'^5 - \lambda e'^6,\quad e^6 = \lambda e'^5 + e'^6.$$  
This basis preserves equations~\eqref{estructure_equations_8} and since $Je^5 = -e^6$, it is an 
adapted basis for the Hermitian structure.  Moreover, $de^5 =0$ and we can suppose that $de^6\neq0$. 
 
Imposing $d^2 e^7 =0$ in~\eqref{estructure_equations_8} we obtain that $c_{12}^7 + c_{34}^7 = 0$ 
and the following system of equations:
\begin{eqnarray*}
c_{13}^6c_{15}^7 + c_{14}^6c_{16}^7 - c_{12}^6c_{36}^7&=&0,\\
c_{14}^6c_{15}^7 - c_{13}^6c_{16}^7 - c_{12}^6c_{35}^7&=&0,\\
c_{12}^6c_{16}^7 - c_{13}^6c_{35}^7 + c_{14}^6c_{36}^7&=&0,\\
c_{12}^6c_{15}^7 + c_{14}^6c_{35}^7 - c_{13}^6c_{36}^7&=&0.
\end{eqnarray*}
If we consider the previous system as homogeneous in variables $c_{ij}^7$, the determinant of the 
associated matrix is simply $((c_{12}^6)^2 + (c_{13}^6)^2 + (c_{14}^6)^2)^2\neq 0$ since $de^6\neq 
0$.  Therefore, $c_{15}^7=c_{16}^7=c_{35}^7=c_{36}^7=0$.

Similarly, we obtain that $c_{12}^8 + c_{34}^8 = c_{15}^8=c_{16}^8=c_{35}^8=c_{36}^8=0$.  But in 
this case, the center should be 4-dimensional, which is a contradiction.

Finally, observe that $de^5,\, de^6\neq 0$ and therefore $e^5, e^6\in [\frg, \frg]$.  Since 
$e_5\notin\mathfrak z$, $[e_i, e_5]\neq 0$ for some $i=1,\ldots, 4$, which means that $[[\frg, 
\frg], \frg]\neq 0$.  However $[[[\frg, \frg], \frg],\frg]=0$.
\end{proof}

\begin{remark}
{\rm A particular case of II.1) in the previous proposition is $\frh_7\times \R$ which appear if 
$de^7$ and $de^8$ are linearly dependent but not zero simultaneously.}
\end{remark}

\

%%%%%%%%%%%%%%%%%%%%%%%%%%%%%
\section{Examples}
%%%%%%%%%%%%%%%%%%%%%%%%%%%%

In this section we will provide two ways of constructing unimodular Lie algebras endowed with
abelian balanced Hermitian structures. In the first one the examples arise from
nilpotent commutative associative algebras, whereas the second method uses semidirect products of
abelian complex Lie algebras.

%%%%%%%%%%%%%%%%%%%
\subsection{Nilpotent commutative associative algebras}
%%%%%%%%%%%%%%%%%%%

Let us recall the following construction, introduced in \cite{BD}. If $A$ is a finite-dimensional
commutative associative algebra, then $\aff(A):=A \oplus A$ is a Lie algebra with the following
bracket:
\begin{equation}\label{bracket-aff} 
 [(a,b),(a',b')]=(0, ab'-ba'). 
\end{equation}
Moreover, the endomorphism $J:\aff(A)\to \aff(A)$ defined by $J(a,b)=(b,-a)$ is an abelian complex
structure, called the \textit{standard} abelian complex structure.
According to \cite{AO}, the Lie algebra $\aff(A)$ is unimodular if and only if $A$ is nilpotent,
and in this case, it follows that $\aff(A)$ is actually a  nilpotent Lie algebra.

So, let us assume from now on that $A$ is a nilpotent commutative associative algebra of dimension $n$. Let us fix
an inner product $\langle \cdot,\cdot \rangle$ on $A$ and extend it to an inner product on $\aff(A)$ by
considering $\langle \cdot,\cdot \rangle \oplus \langle \cdot,\cdot \rangle$. It is readily seen
that $\langle J\cdot, J\cdot \rangle = \langle \cdot,\cdot \rangle$, so that $(J,F)$ is a Hermitian
structure on $\aff(A)$, where $F(\cdot, \cdot):=\langle \cdot, J \cdot \rangle$, as usual. Let  $\{e_1,\ldots, e_n\}$ denote
an orthonormal basis of $A$, and set $u_i:=(e_i,0),\, v_i:=(0,e_i)$, $i=1,\ldots, n$, so that
$Ju_i=-v_i$, thus $\{u_1,v_1,\ldots, u_n, v_n\}$ is an adapted basis of $\aff(A)$.

According to Lemma~\ref{J-bon}, the Hermitian structure is balanced if and only if $\sum_{i=1}^n
([Ju_i,u_i]+[Jv_i,v_i])=0$, but due to \eqref{bracket-aff}, this condition is equivalent to
$\sum_{i=1}^n e_i^2=0$. This proves the following result:

\begin{proposition}\label{nilpotent-alg}
Let $A$ be a nilpotent commutative associative algebra of dimension $n$ with an inner product, and consider the
Hermitian structure $(J,F)$ on $\aff(A)$ as above, where $J$ is the standard abelian complex
structure. Then $(J,F)$ is balanced if and only if there exists an orthonormal basis $\{e_1,\ldots,
e_n\}$ of $A$ such that $\sum_{i=1}^n e_i^2=0$.
\end{proposition}

\

\begin{example}
{\rm
According to \cite{dG} (see also \cite{GR}), there are $4$ isomorphism classes of $3$-dimensional
nilpotent commutative associative algebras. However, it is easy to verify that only one of them
satisfies the conditions of Proposition~\ref{nilpotent-alg}. We will denote this algebra by $A_1$;
it has a basis $\{e_1,e_2,e_3\}$ such that $e_1^2=-e_3,\, e_2^2=e_3$, and we will define an inner
product on $A_1$ by declaring this basis orthonormal. The corresponding nilpotent Lie algebra
$\aff(A_1)$ is isomorphic to $\frh_5\times \R$, and its abelian balanced Hermitian structure $(J,F)$ corresponds to one of the
balanced structures on this Lie algebra appearing in \cite{UV}, with the complex structure
$\tilde{J}^-$.}
\end{example}

\

\begin{example}
{\rm
In dimension $4$, there are $3$ (up to isomorphism) nilpotent commutative associative algebras that
satisfy the conditions of Proposition~\ref{nilpotent-alg} (see \cite{dG}), which we denote $B_1$,
$B_2$ and $B_3$.

\begin{itemize}

\item[(i)] The algebra $B_1$ is isomorphic to the direct product
$B_1=A_1\times \R$, and accordingly $\aff(B_1)=\aff(A_1)\times \R^2 \cong (\frh_5\times \R)\times
\R^2$, a direct product of Hermitian Lie algebras. 

\medskip

\item[(ii)]  The algebra $B_2$ has a basis $\{e_1,e_2,e_3, e_4\}$ such that $e_1^2=e_4,\,
e_2^2=e_4, \, e_3^2=-2e_4$. For any $\lambda >0$, we will define an inner product $\langle
\cdot,\cdot \rangle_\lambda$ on $B_2$ by declaring the basis $\{e_1,e_2,e_3,\lambda^{-1} e_4\}$
orthonormal. Equivalently, we can consider a new orthonormal basis $\{e_1,e_2,e_3, e_4\}$ where now
the product is given by $e_1^2=\lambda e_4,\, e_2^2=\lambda e_4, \, e_3^2=-2\lambda e_4$, $\lambda
>0$. The corresponding nilpotent Lie algebra $\aff(B_2)_\lambda$ has an adapted basis
$\{u_1,v_1,\ldots, u_4,v_4\}$ such that 
\[ [u_1,v_1]=\lambda v_4,\quad [u_2,v_2]=\lambda v_4,\quad [u_3,v_3]=-2\lambda v_4, \]
and the fundamental $2$-form is given in terms of the dual basis $\{u^i, v^i\}$ by $F=\sum_{i=1}^4
u^i\wedge v^i$. Clearly, this Lie algebra is isomorphic to $\frh_7\times \R$ for any $\lambda>0$.
Moreover, the associated Hermitian metrics $g_\lambda=F(J\cdot,\cdot)$ are pairwise non-isometric,
since the corresponding scalar curvature is $-3\lambda^2$. Therefore we obtain a curve
$(J_2,F_\lambda),\, \lambda>0$, of non-equivalent abelian balanced Hermitian structures on
$\frh_7\times \R$, where $J_2$ is the abelian complex structure given in the proof of
Corollary~\ref{Heis}.

\medskip

\item[(iii)]  The algebra $B_3$ has a basis $\{e_1,e_2,e_3, e_4\}$ such that $e_1^2=-e_3,\,
e_2^2=e_3, \, e_1e_2=e_2e_1=e_4$, and we will define an inner product on $B_3$ by declaring this
basis orthonormal. The corresponding nilpotent Lie algebra $\aff(B_3)$ has an adapted basis
$\{u_1,v_1,\ldots, u_4,v_4\}$ such that
\[ [u_1,v_1]=-v_3,\quad [u_2,v_2]=v_3,\quad [u_1,v_2]=[u_2,v_1]=v_4, \]
and the fundamental $2$-form is given in terms of the dual basis $\{u^i, v^i\}$ by $F=\sum_{i=1}^4
u^i\wedge v^i$.  Consider the new basis given by: 
\[ h^1 = u^1,\,\, h^2 = -u^2,\,\, h^3 = v^1,\,\, h^4 = v^2,\,\, h^5 = u^3,\,\, h^6 = u^4,\,\, h^7 =
v^3,\,\, h^8 = -v^4.\]
Now, it is clear that $\aff(B_3)$ is isomorphic to $\frg_2$ in Proposition~\ref{4-dim-center}. 
\end{itemize}
}
\end{example}

\

%%%%%%%%%%%%%%%%%%%
\subsection{Semidirect products of abelian complex Lie algebras}
%%%%%%%%%%%%%%%%%%%

Let $m,n\in\N$ and consider a representation
$\pi:\C^m\to\operatorname{End}(\C^n)\cong\mathfrak{gl}(n,\C)$ such that $\tr(\pi(x))=0$ for any
$x\in\C^m$, i.e., $\pi(\C^m)\subset \mathfrak{sl}(n,\C)$. We can therefore build a semidirect
product and obtain in this way the complex Lie algebra $\frg:=\C^m\ltimes_\pi \C^n$, which is
$2$-step solvable and unimodular. We will denote by $\tilde{\frg}$ its realification,
$\tilde{\frg}:=\frg_\R$, which is also $2$-step solvable and unimodular. Choosing some identifications
$\C^m\equiv\R^{2m}$ and $\C^n\equiv\R^{2n}$, the representation $\pi$ induces a representation
$\tilde{\pi}:\R^{2m}\to \mathfrak{sl}(n,\C)\hookrightarrow\mathfrak{sl}(2n,\R)$, and therefore we
have the decomposition $\tilde{\frg}=\R^{2m}\ltimes_{\tilde{\pi}} \R^{2n}$. Note that the Lie
algebra $\tilde{\frg}$ does not depend on the chosen identifications, it depends only on the representation $\pi$.

Multiplication by $i$ on $\frg$ induces a canonical bi-invariant complex structure $I$ on
$\tilde{\frg}$, that is $I^2=-\operatorname{id}$ and $[Ix,y]=[x,Iy]=I[x,y]$ for any
$x,y\in\tilde{\frg}$. By suitably modifying the bi-invariant complex structure $I$, we will obtain
an abelian complex structure on 
$\tilde{\frg}$. Indeed, let us define $J:\tilde{\frg}\to\tilde{\frg}$ by
%\begin{equation} 
\[ J|_{\R^{2m}}=-I, \qquad J|_{\R^{2n}}=I,  \]
%\end{equation}
and compute, for $x_j\in\R^{2m},\,y_j\in\R^{2n}$, $j=1,2$:
\[ [J(x_1+y_1),J(x_2+y_2)]= [-Ix_1+Iy_1, -Ix_2+Iy_2]=[x_1,y_2]+[y_1,x_2]=[x_1+y_1,x_2+y_2],\]
thus $J$ is abelian. Note that both $\R^{2m}$ and $\R^{2n}$ are $J$-invariant.

Now let us choose an inner product $\langle \cdot,\cdot \rangle$ such that it is Hermitian with
respect to $J$ and the two factors are orthogonal. There exists an adapted basis
$\{e_1,\ldots,e_{2m+2n}\}$ of $\tilde{\frg}$ such that $\{e_1,\ldots,e_{2m}\}$ is an adapted basis
of $\R^{2m}$ and $\{e_{2m+1},\ldots,e_{2m+2n}\}$ is an adapted basis of $\R^{2n}$.
Clearly, $\sum_{i=1}^{2m+2n} [Je_i,e_i] =0$, and therefore, according to Lemma~\ref{J-bon}, $(J,F)$
is an abelian balanced Hermitian structure on $\tilde{\frg}$, where $F$ is the associated
fundamental $2$-form.

To sum up, we state the following

\begin{proposition}
The realification of the unimodular complex Lie algebra $\C^m\ltimes_\pi\C^n$, where $\pi:\C^m\to
\mathfrak{sl}(n,\C)$ is a representation, admits abelian balanced Hermitian structures.
\end{proposition}

\medskip

Let us consider in more detail the case $m=1$. In this situation, $\frg=\C\ltimes_\pi \C^n$ is
called an almost abelian complex Lie algebra, and the action of $\C$ on $\C^n$ is determined by a
single traceless operator $M:=\pi(1)\in\mathfrak{sl}(n,\C)$. Let us consider now its realification
$\tilde{\frg}=\R^{2}\ltimes_{\tilde{\pi}}\R^{2n}$. If $\{A,\,B\}$ is a basis of $\R^2$ with $IA=B$,
then the associated real representation $\tilde{\pi}:\R^{2}\to\mathfrak{sl}(2n,\R)$ is characterized by
$\tilde{\pi}(A)$ and $\tilde{\pi}(B)$. If $\tilde{M}$ denotes the image of $M$ under the embedding
$\mathfrak{sl}(n,\C)\hookrightarrow\mathfrak{sl}(2n,\R)$ determined by the identification
$\C^n\equiv\R^{2n}$, then clearly $\tilde{\pi}(A)=\tilde{M}$ and
$\tilde{\pi}(B)=I\tilde{M}=\tilde{M}I$. Therefore the Lie bracket on the Lie algebra
$\tilde{\frg}=\R^2\ltimes_{\tilde{\pi}}\R^{2n}$ is described by 
$[A,X]=\tilde{M}X, \, [B,X]=I\tilde{M}X$ for any $X\in \R^{2m}$; its abelian complex structure is given 
by $JA=-B,\, JX=IX$ for any $X\in \R^{2m}$.

\begin{remark}\label{isomorphic}
{\rm It is well known that two almost abelian Lie algebras $\C\ltimes_M\C^n$ and
$\C\ltimes_{M'}\C^n$ are isomorphic if and only if $M'=cPMP^{-1}$ for some non-zero $c\in\C$ and
$P\in GL(n,\C)$. Clearly, the realification of two isomorphic complex Lie algebras are isomorphic as
real Lie algebras.}
\end{remark}

\begin{example}
{\rm 
\begin{itemize}
\item[(1)]If we consider $\frg=\C\ltimes_M\C^2$, then taking into account the canonical Jordan form
of a $2\times 2$ traceless complex matrix and Remark~\ref{isomorphic}, we only
have to consider two possibilities for $M$, either
\[ (i)\, M=\begin{pmatrix} 0 & 0 \\ 1 & 0 \end{pmatrix}, \quad \text{or} \quad 
(ii)\, M=\begin{pmatrix} 1 & 0 \\ 0 & -1 \end{pmatrix}.\]
In case (i), $\frg$ is a $2$-step nilpotent Lie algebra, and its realification $\tilde{\frg}$ is
isomorphic to the Lie algebra with structure equations 
\begin{equation}\label{heis_real}de^i=0,\, i=1,2,3,4,\quad de^5 = e^{13}+e^{42},\quad 
de^6=e^{14}+e^{23},\end{equation} 
i.e., it is the realification of the complex Heisenberg Lie 
algebra. It was shown in \cite{UV} that any abelian
complex structure on this Lie algebra admits balanced metrics. In case (ii), we reobtain the
unimodular solvable Lie algebra from Remark~\ref{nil}, with the same abelian balanced Hermitian
structure.

\item[(2)]We consider now $\frg=\C\ltimes_M\C^3$, where $M$ is a $3\times 3$ traceless complex matrix.
The possible canonical Jordan forms for such a matrix are
\[ (i)\, M=\begin{pmatrix} \lambda_1 & 0 & 0  \\ 0 & \lambda_2 & 0 \\ 0 & 0 & \lambda_3
\end{pmatrix}, \quad (ii)\, M=\begin{pmatrix} \lambda & 0 & 0 \\ 1 & \lambda & 0 \\ 0 & 0 &
-2 \lambda \end{pmatrix}, \quad (iii)\, M=\begin{pmatrix} 0 & 0 & 0 \\ 1 & 0 & 0 \\ 0 & 1 & 0
\end{pmatrix},\]
where $\lambda,\lambda_j\in\C$ with $\sum_{j=1}^3 \lambda_j=0$. From this, we easily see that the only
nilpotent examples are obtained in case (ii), for $\lambda=0$, and in case (iii). In the former
case, the Lie algebra $\frg$ is $2$-step nilpotent, and $\tilde{\frg}$ is isomorphic to the direct
product of the Lie algebra underlying equations~\eqref{heis_real} with $\R^2$. In the latter case, $\frg$ is the only
$4$-dimensional $3$-step nilpotent complex Lie algebra, and its realification $\tilde{\frg}$ has a
basis $\{e_1,\ldots,e_8\}$ adapted to an abelian balanced Hermitian structure satisfying:
\[ \begin{cases}
    de^i=0,\quad i=1,2,3,4, \\
    de^5=-e^{13}-e^{24},\\
    de^6=-e^{14}+e^{23},\\
    de^7=-e^{15}-e^{26},\\
    de^8=-e^{16}+e^{25}.
   \end{cases} \]
This is a particular case of Proposition~\ref{prop:center4_dim8}, case II.2). According to 
Theorem~\ref{holonomia}, the holonomy algebra of the Bismut connection $\nabla$ is
contained in $\mathfrak{su}(3)$, since the center $\frz=\text{span}\{e_7,e_8\}$ is
$2$-dimensional. An example of a non-nilpotent example where this holonomy reduction occurs is in
case (i), when $\lambda_3=0$ and $\lambda_2=-\lambda_1$. According to Remark~\ref{isomorphic}, we
may assume $\lambda_1=-\lambda_2=1$, and therefore $\tilde{\frg}$ is isomorphic to the direct
product of $\R^2$ and the unimodular solvable Lie algebra from Remark~\ref{nil}.

\item[(3)] If we consider now almost abelian complex Lie algebras of the form $\frg=\C\ltimes_M\C^4$ with
$M$ a $4\times 4$ traceless complex matrix, we will be able to provide examples of
$10$-dimensional indecomposable solvable non-nilpotent Lie algebras admitting an abelian balanced
Hermitian structure whose holonomy algebra is contained properly in $\mathfrak{su}(5)$. Indeed, for
any $\lambda=a+bi\in \C-\{0\}$, consider the matrix
\[ M_\lambda = \begin{pmatrix}
                0 & 0 & 0 & 0 \\
                1 & 0 & 0 & 0   \\
                0 & 0 & \lambda & 0 \\
                0 & 0 & 0 & -\lambda
      \end{pmatrix}.
\]
In this case $\frg_\lambda=\C\ltimes_{M_\lambda}\C^4$ is a $5$-dimensional solvable non-nilpotent
complex Lie algebra, which can be seen to be indecomposable, and its realification
$\tilde{\frg}_\lambda$ has a basis $\{e_1,\ldots,e_{10}\}$ adapted to an abelian balanced Hermitian
structure satisfying:
\[ \begin{cases}
    de^i=0,\quad i=1,2,3,4, \\
    de^5=-e^{13}-e^{24},\\
    de^6=-e^{14}+e^{23},\\
    de^7=-a(e^{17}+e^{28})-b(e^{18}-e^{27}),\\
    de^8=b(e^{17}+e^{28})-a(e^{18}-e^{27}),\\
    de^9=a(e^{19}+e^{2,10})+b(e^{1,10}-e^{29}),\\
    de^{10}=-b(e^{19}+e^{2,10})+a(e^{1,10}-e^{29}).
   \end{cases} \]
According to Theorem~\ref{holonomia}, the holonomy algebra of the Bismut connection $\nabla$ is
contained in $\mathfrak{su}(4)$, since the center $\frz=\text{span}\{e_5,e_6\}$ is
$2$-dimensional.

\end{itemize}
}
\end{example}

\

{\itshape Acknowledgments.} The authors would like to warmly thank the kind 
hospitality during their respective stays at Departamento de Matem\'aticas (Universidad de 
Zaragoza, Spain) and Facultad de Matem\'atica, Astronom\'ia y F\'isica (Universidad Nacional de 
C\'ordoba, Argentina). This work has been partially supported by CONICET, ANPCyT and SECyT-UNC 
(Argentina) and Project MICINN MTM2011-28326-C02-01 (Spain).

%%%%%%%%%%%%%%%%%%%%%%%%%%%%%%%%%%%%%%%%%%

\end{document}